\documentclass[a4paper,10pt]{amsart}

\usepackage[active]{srcltx}
\usepackage{mathrsfs}
\usepackage[all]{xy}
\usepackage{calrsfs}
\usepackage{times}
\usepackage[scaled=0.92]{helvet}
\usepackage[greek,english]{babel}
\usepackage[iso-8859-7]{inputenc}

\usepackage{amsfonts}
\usepackage{amssymb}
\usepackage[usenames]{color}
\usepackage[dvistyle]{todonotes}

\newtheorem{theorem}{Theorem}   

\newtheorem{lemma}[theorem]{Lemma}
\newtheorem{proposition}[theorem]{Proposition}

\theoremstyle{definition}
\newtheorem{definition}[theorem]{Definition}
\newtheorem{remark}[theorem]{Remark}
\newtheorem{example}[theorem]{Example}

\definecolor{MyDarkBlue}{rgb}{0,0.08,0.60}

\newcommand{\Z}{{\mathbb{Z}}}
\newcommand{\lc}{\left\lceil}
\newcommand{\rc}{\right\rceil}
\newcommand{\lf}{\left\lfloor}
\newcommand{\rf}{\right\rfloor}

\renewcommand{\mod}{{\;\rm mod}}

\title{Representations of cyclic groups in positive characteristic and Weierstrass semigroups}

\date{\today}

\author{Sotiris Karanikolopoulos }
\address{Department of Mathematics,  University of the Aegean \\
Karlovassi, 83200 Samos, Greece}
\email{mathm03005@aegean.gr}

\author{Aristides Kontogeorgis}
\address{Department of Mathematics, University of Athens \\
Panepistimioupolis, 15784 Athens, Greece
}
\email{kontogar@math.uoa.gr}

\begin{document}

\begin{abstract}
We study the $k[G]$-module structure of the space of holomorphic differentials of 
a curve defined over an algebraically closed field of positive characteristic, 
for a cyclic group $G$ of order $p^\ell n$. We also study the relation to the
Weierstrass semigroup for the case of Galois Weierstrass points. 
\end{abstract} 

\thanks{{\bf keywords:} Automorphisms, Curves, Differentials, Numerical Semigroups.
 }

\maketitle 
\section{Introduction}
Let $X$ be a projective nonsingular curve, defined over an 
algebraically closed field $K$ of positive characteristic $p$. 
The study of the curve $X$ is equivalent to the study of the 
corresponding function field $F$.

An open question in positive characteristic is the determination of the 
Galois module structure of the space of holomorphic differentials of $X$. 
This problem is still open and only some special cases are known 
\cite{Tamagawa:51},\cite{Val:82},\cite{Salvador:00},\cite{csm},\cite{borne06}
where restrictions are made either on the ramification type or on the group structure of
$G$. 
R.Valentini and M. Madan in \cite{vm} computed the Galois module structure 
of the space of holomorphic differentials for the case of a cyclic group action 
$G$, where $G$ was a cyclic $p$-group of order either prime to $p$ or a power 
of $p$. One of the aims of this paper is to extend the result of Valentini Madan to 
the more general case of a cyclic group that has order $p^\ell n$, $(n,p)=1$.
We will characterize the  indecomposable summands $V(\lambda,k)$ (see 
section \ref{notation} for a  precise definition
 in terms of the Jordan indecomposable
blocks of the generator) and we will decompose the  space $V$ of 
holomorphic differentials as:
\begin{equation} 
 \label{decompo-1}
 V:=\bigoplus_{\lambda=0}^{n-1} \bigoplus_{k=1}^{p^\ell} V(\lambda,k)^{d(\lambda,k)}.
\end{equation}
 The numbers $d(\lambda,k)$ will be described 
in terms of the ramification of the extension $F/F^G$ in theorem \ref{main-th}.  

The $G$--module structure is expressed in terms of the \textit{Boseck}  invariants.
 These are invariants introduced by Boseck \cite{boseck} coming from the construction of bases
of holomorphic differentials. 
The Boseck invariants have  rich connections with other subjects in the literature:
computation of Weierstrass points, \cite{boseck}, \cite{garciaa-s}, \cite{garciaelab}; the computation of 
the rank of the Hasse--Witt matrix,
 \cite{Madden78}; the classification of curves with certain rank of the Hasse--Witt matrix \cite{CaSa};
the study of the Artin--Schreier (sub)extensions of rational functions fields, \cite{Val-Mad:80},etc. 
Here we choose to focus only on the $G$ module structure as well as on the structure of the Weierstrass semigroup that is attached to 
a ramified point.
  
The complicated  notation needed in order to state the main results prevents us from presenting 
our main theorem here.

The paper is organized as follows: In section \ref{notation} we introduce a 
notation for the places that are ramified in extension $F/F^P/F^G$ and give 
a filtration of the module of holomorphic differentials used in the computations. 
Next section is devoted to dimension  computations  with the aid of 
Riemann-Roch formula. In the final section we see the relation to the 
Weierstrass semigroup. We tried to relate our results to known results 
in the literature. This way we discovered an inaccuracy in the work of Boseck
\cite{boseck} in the case of a $\Z/p\Z$-extension of the rational function field 
ramified above one point. 
Finally we extend results from characteristic zero relating the Galois module structure 
of the space of holomorphic differentials and the Weierstrass semigroup attached
 to a ramified point.

\section{Notation} \label{notation}

 Let $G=\langle g \rangle$ be a cyclic subgroup of automorphisms acting on the 
space of holomorphic differentials $V:=H^0(X,\Omega_X)$. 
The group $G$ can be written as a direct product of a group $T=\langle g^{p^\ell} \rangle$ of order $n$
and a  cyclic $p$-group $P=\langle g^{n} \rangle$.  
We consider the tower of function  fields $F/F^P/F^G$.
Let $np^\ell=|G|$, $(n,p)=1$ and consider a primitive $n$-th root of unity $\zeta_n\in K$.
By Jordan decomposition theory we see that we can decompose  $V$ as a direct sum 
of $K[G]$-modules $V(\lambda,k)$. The modules 
$V(\lambda,k)$ are $k$-dimensional $K$-vector spaces with basis $\{v_1,\ldots,v_k\}$
and action given by 
\begin{equation} \label{staction}
 g v_i =\zeta_n^\lambda v_i+v_{i+1} \mbox{ for all } 1 \leq i \leq k-1
\end{equation}
and 
\[
 g v_{k} =\zeta_n^\lambda v_{k}.
\]
The action of the generator $g$ on $V(\lambda,k)$ is given in 
terms of the matrix:
\[
A:=
 \begin{pmatrix}
  \zeta_n^\lambda & 1            & 0 & \cdots & 0 \\
    0       & \zeta_n^\lambda    & 1 & \ddots & \vdots \\
    \vdots  &  \ddots       &\ddots & \ddots & 0 \\
    0       & \cdots        &  0     & \zeta_n^\lambda & 1\\ 
    0       & \cdots        &   0 & 0    &  \zeta_n^\lambda
 \end{pmatrix}.
\]
Observe that for a cyclic group $G$ of order $np^\ell$ generated by $g$ the module $K[G]$ 
can be decomposed as $K[G]=\oplus_{\lambda=0}^{n-1} V(\lambda,p^\ell)$. Indeed, 
the characteristic polynomial of $g$ acting on $K[G]$ is up to $\pm 1$ equal to $x^{np^\ell}-1=(x^n-1)^{p^\ell}$, 
and every root of unity in $K$ appears as a character in $K[G]$.
\begin{remark}
 The  indecomposable $K[P]$-modules of a cyclic $p$-group of order $p^\ell$ and with generator 
$\sigma$ are given by the quotients 
$K[P]/(\sigma-1)^k$, where $k=1,\ldots,p^\ell$ \cite{vm}. In our notation these are the modules
$V(0,k)$ i.e. the indecomposable  Jordan forms of {dimension $k$}.
\end{remark}

%
\begin{proposition}
The indecomposable $K[G]$-module $V(\lambda,k)$ seen as a $K[T]$-module
is a direct sum of $k$ characters of the form $\zeta \mapsto \zeta^{p^\ell n}$.
The module $V(\lambda,k)$ seen as $K[P]$-module is indecomposable and isomorphic 
to the module $K[P]/(\sigma-1)^{k}$.   
\end{proposition}
\begin{proof}
We will use the following idea: The action of $G$ on the indecomposable summand 
$V(\lambda,k)$ is described by the action of the generator $g$ of $G$. We would like to 
view the module $V(\lambda,k)$ as a $P$ and $T$ module respectively. A generator for the $T$ group 
is given by 
$g^{p^\ell}$. Write the matrix $A$ as $A=\mathrm{diag}(\zeta^\lambda)+N$ where $N$
 is a nilpotent $k \times k$
matrix with $k \leq p^\ell$. Therefore, the generator $g^{p^\ell}$ of the $T$ group 
is given by the matrix $A^{p^\ell}=\mathrm{diag}(\zeta^{\lambda p^\ell})$. This means that 
$V(\lambda,k)$ seen as a $T$ module is decomposed as a direct sum of $k$
characters of the form $\zeta \mapsto \zeta^{\lambda p^\ell}$. 
Since $(p^\ell,n)=1$ raising an $n$-th root of unity to the $p^\ell$-power is an 
automorphism of the group of $n$-th roots of one. 

On the other hand  the action of the generator $g^n$ on the module $V(\lambda,k)$ is 
given by the $n$-th power of $A$. We observe first that that all eigenvalues of 
$A^n$ are $1$. We will prove that $A^n$ is similar to the matrix $\mathrm{Id}+N$,
i.e. a Jordan indecomposable block. Since 
all eigenvalues of $A^n$ are $1$ the characteristic polynomial of $A^n$ is 
$(x-1)^k$. The minimal polynomial of $A^n$ is $(x-1)^d$ for some integer $1 \leq d \leq k$. 
Since $A$ is an indecomposable Jordan block the minimal polynomial 
of $A$ is $(x-\zeta^\lambda)^k$. On the other hand, since $(A^n-1)^d=0$ we 
have that $(x-\zeta^\lambda)^k$ divides $(x^n-1)^d$ and this is possible only if 
$d=k$. This implies that $A^n$ is similar to an indecomposable Jordan form of dimension 
$k$. 
\end{proof}

\subsection{Fields and ramification}
\label{FR}

We will introduce some notation on the ramification 
places in the extensions $F/F^P$ and $F/F^G$.
Let us denote by $\bar{P}_1,\ldots,\bar{P}_s$ the places of $F^P$ that are ramified in $F/F^P$.
The places of $F$ that are above $\bar{P}_i$ will be denoted by $P_{i,\nu}$, $1\leq \nu \leq p^\ell/e_i$,
where $e_i=p^{\epsilon_i}$ is the common ramification index  $e(P_{i,\nu}/\bar{P}_i)$.

The different  $\mathrm{Diff}(F/F^P)$ is supported at the places $P_{i,\nu}$ while the discriminant 
$D(F/F^P)$ is  supported at the places $\bar{P}_1,\ldots,\bar{P}_s$. 
Let us denote the different exponent at each ramified place $P_j$ by $\delta_j$. The 
discriminant is then computed:
{
\[
 D(F/F^P)=\sum_{j=1}^s p^{\ell-\epsilon_j} \delta_j \bar{P}_j,
\]
while the different is given by
\[
\mathrm{Diff}(F/F^P)=\sum_{j=1}^s  \delta_j \sum_{\nu }{P}_{j,\nu}.
\]}
The cyclic group extension $F^P/F^G$ is a Kummer extension
with Galois group $T$
 and it is defined by an equation of the form:
\begin{equation} \label{def-eq-cyc}
F^P=F^G(y), \;\;\; y^n=b, \;\;\;b \in F^G.                                                                                               
\end{equation}

Let $\bar{Q}_1,\ldots,\bar{Q}_t$ be the places of $F^G$ that are ramified in extension 
$F^P/F^G$. We define  $Q_{i,\nu}$, $1\leq \nu \leq n/e_i'$ to be the places of $F^P$ 
which  are above  $\bar{Q}_i$, where $e_i'$ denotes the common ramification index, 
$e_i'=e(Q_{i,\nu}/\bar{Q}_i)$. 

Assume that the set of places $\{ \bar{Q}_1,\ldots, \bar{Q}_{t_0}\}$ extend to places 
$Q_{i,\nu}$ of $F^P$ that do not ramify on $F/F^P$ and that each place 
$\bar{Q}_i$ of the   places
$\{
\bar{Q}_{t_0+1},\ldots,\bar{Q}_t \}$ extends to places $Q_{i,\nu}$ that ramify in $F/F^P$.
The total number of places of the form 
$Q_{i,\nu}$ $t_0+1\leq i \leq t$ equals
\[
 s_0:=\sum_{i=t_0+1}^t \frac{n}{e_i'}= |\{Q_{i,\nu}: t_0+1\leq i \leq t, 1\leq \nu \leq n/e_i'  \}|.
\]
We enumerate the places $\bar{P}_i$ such that $\{\bar{P}_{s_0+1},\ldots, \bar{P}_{s} \}$
do not ramify in $F^P/F^G$ and 
$\{P_1,\ldots,P_{s_0}\}=\{Q_{i,\nu}: t_0+1\leq i \leq t, 1\leq \nu \leq n/e_i'  \}$.

\[
\xymatrix{ 
F \ar@{-}[d]_{P} & & &P_{1,\mu} \ar@{-}[d]_{e_1} \ar@{..}[r] & P_{t,\mu} \ar@{-}[d]_{e_{t}} & 
 P_{s_0+1,\nu} \ar@{-}[d]_{e_{s_0+1}} \ar@{..}[r]   
 &P_{s,\nu} \ar@{-}[d]_{e_s} \\
F^{P} \ar@{-}[d]_T & Q_{1,\nu} \ar@{-}[d]_{e_1'} \ar@{..}[r]  &Q_{t_0,\nu} \ar@{-}[d]_{e_{t_0}'} &
Q_{t_0+1,\nu} \ar@{-}[d]_{e_{t_0+1}'} \ar@{..}[r]  & Q_{t,\nu} \ar@{-}[d]_{e_t'} &
 \bar{P}_{s_0+1}  \ar@{..}[r]   & \bar{P}_s \\
F^G &   \bar{Q}_1  \ar@{..}[r]   & \bar{Q}_{t_0}  & \bar{Q}_{t_0+1}  \ar@{..}[r]  & \bar{Q}_{t} 
}
\]
We can select $b$ in eq. (\ref{def-eq-cyc}) such that \cite[sec. 2]{vm}
\begin{equation} \label{div-b}
 \mathrm{div}_{F^G}(b)=n A +\sum_{i=1}^t {\phi_i}\bar{Q}_i,
\end{equation}
where $0<\phi_i<n$, $A$ is a divisor of $F^G$. 
The ramification indices are given by   $e_i'=n/(n,\phi_i)$, and the discriminant is given by
\begin{equation} \label{diff-tame}
 D(F^P/F^G)=\sum_{i=1}^t \left(n-\frac{n}{e_i'} \right)\bar{Q}_i.
\end{equation}
We also define $\Phi_i=\phi_i/(n,\phi_i)$.

\subsection{Modules}

Let us now focus on the cyclic $p$-group extension $F/F^P$. 
The $G$-module structure on holomorphic differentials on a cyclic $p$-group is studied by 
R. Valentini and M. Madan in \cite{vm}.
Let $\sigma:=g^n$ be a generator of the cyclic group $P$. 
Recall that $V$ denotes the set of holomorphic differentials.
Following the article of Valentini Madan  we consider the set of subspaces
 $V^i \subset V$ defined by  
\[
 V^i:=\{ \omega \in V : (\sigma-1)^i \omega =0 \} \mbox{ for } i=0,\ldots, p^\ell.
\]
We compute that (the set $\{v_1,\ldots,v_m\}$ is  the basis of $V(\mu,m)$ given in eq. (\ref{staction})). 
\[
 V^{i+1} \cap V(\lambda,m) =
\left\{
\begin{array}{ll}
 V(\lambda,m) & \mbox{ if } m\leq i+1 \\
 \langle  v_{m-i},\ldots,v_{m} \rangle & \mbox{ if } m >i+1
\end{array}
\right.
\]
Since $G$ is a commutative group there is a well defined action of
 $T=\langle g^{p^\ell} \rangle$ on the 
quotient space $V^{i+1}/V^{i}$ and the natural map  
\[
 V^{i+1} \rightarrow V^{i+1}/V^i,
\]
is $T$-equivariant. The images of the spaces 
 $V^{i+1} \cap V(\lambda,m)$ 
under this map are $0$ for $m\leq i$ and are one dimensional if $m>i$.

The space $V^{i+1}/V^i$ is decomposed into characters of the group $T$. 
Let $d(\lambda,k)$ be the number of $V(\lambda,k)$ blocks in $V$. 
Let $c(\lambda,i)$, $0\leq i \leq p^\ell-1$ be the number of characters of the form $g \omega =\zeta^\lambda \omega$ in 
$V^{i+1}/V^i$.


We have that 
\[
 c(\lambda,i)=\sum_{k \geq i+1} d(\lambda,k).
\]
Therefore 
\begin{eqnarray}\label{dk}
d(\lambda,p^\ell) & = & c(\lambda,p^\ell-1) \\
d(\lambda,k) & = & c(\lambda,k-1)-c(\lambda,k). \nonumber
\end{eqnarray}

\begin{lemma} \label{properties}
 There is a basis $\{w_0,\ldots, w_{p^\ell-1}\}$ of $F$ over $F^P$ 
such that:
\begin{enumerate}
\item
For $0\leq k \leq p^\ell-1$ with $p$-adic expansion 
$k=a_1^k+a_2^k p+\cdots+a_\ell^k p^{\ell-1}$, 
we have 
\[
 (\sigma-1)^kw_k=a_1^k! a_2^k!\cdots a_n^k! w_k.
\]
 \item 
Every $\omega \in V$ can be written as 
\[
 \omega =\sum_{\nu=0}^{p^\ell-1} c_\nu w_\nu dx
\]
with $x,c_\nu \in F^P$ and with the additional 
property that 
\[
 \omega \in V^{i} \Leftrightarrow c_i=c_{i+1}=\cdots=c_{p^\ell-1}=0.
\]
\item There are numbers $\Phi(\mu,j)$ prime to $p$ such that 
\[
 v_{P_{\mu,\nu}}(w_k)=-\sum_{j=1}^\ell a_j^k \Phi(\mu,j) p^{\ell-j}.
\]

\end{enumerate}

\end{lemma}
\begin{proof}
 The definition of the basis is given in \cite[p.108]{vm}
while the second assertion is proved in the same article in the proof 
of theorem 1.  The existence of the numbers $\Phi(\mu,j)$ follows by the 
construction of the extension $F/F^P$ in terms of successive Artin-Schreier 
extensions (see \cite[sec. 1]{vm}).
\end{proof}

Define the integers:
\begin{equation}\label{defnu}
 \nu_{\mu,k}:=\lf  \frac{\delta_\mu+v_{P_{\mu,\nu}}(w_k)}{e_\mu}\rf.
\end{equation}
Notice that the valuation  $v_{P_{\mu,\nu}}(w_k)$ does not depend on the selection of 
the place $P_{\mu,\nu}$ over $\bar{P}_\mu$.

Let $\tau=g^{p^\ell}$ be a generator of the cyclic group $T$. Assume that $\tau y=\zeta_n^r y$.
For each $\lambda=0,\ldots,n-1$ we select $0\leq \alpha_\lambda \leq n-1$ such that 
\begin{equation} \label{akdef}
 r \alpha_\lambda \equiv \lambda \mod n.
\end{equation}

\begin{definition}\label{defBoseck}
Define the Boseck invariants:
\[
 \Gamma_{k,\lambda}:=
\sum_{i=1}^{t} \left\langle -\frac{\alpha_\lambda \Phi_i}{e_i'}\right\rangle+
\sum_{j=t_0+1}^{t}
\lf
\left\langle
\frac{\alpha_\lambda \Phi_j -1}{e_j'}
\right\rangle
+
\frac{\nu_{j,k}}{n}
\rf
+
\sum_{\mu=1}^{s-t+t_0}
\lf
\frac{\nu_{\mu,k}}{n}
\rf
\].
%
\end{definition}

\begin{remark}
 If $n=1$, then 
\[
 \Gamma_{k,\lambda}= \Gamma_{k}= \sum_{\mu=1}^s \nu_{\mu,k}.
\]
This is the Boseck invariant for the $p$ cyclic case, see \cite{boseck} and  \cite{karan}.

If $p^\ell=1$ then 
\[
\Gamma_{k,\lambda}= \Gamma_{\lambda}= \sum_{j=1}^{t}
\left\langle-
\frac{\alpha_\lambda \Phi_j}{e_j'}
\right\rangle.
\]
This is the Boseck invariant for the the cyclic tame case. 
These invariants coincide with the ones introduced by \cite{boseck}, and used by \cite{karan},  after letting  $r=1$,
to eq. (\ref{akdef}) (this can be done without loss of the generality).
\end{remark}
In next section we will prove the following:
\begin{proposition} \label{pro3}
Recall that $\epsilon_i$ are integers such that $p^{\epsilon_i}=e(P_{i,\nu}/\bar{P}_i)$.
Consider the integer  $r=\ell-\max{\epsilon_i}$.  
{ For $0\leq k< p^\ell-p^r$}, we have
\[
 c(\lambda,k)=g_{F^G}-1+ \Gamma_{k,\lambda} +\Lambda_{k,\lambda}.
\]
The integer $\Lambda_{k,\lambda}$ is given by the following rule:
If
$
\Gamma_{k,\lambda}=0
$
then  $\Lambda_{k,\lambda}=1$. In all other cases $\Lambda_{k,\lambda}=0$.


{
For  $ p^\ell-p^r \leq k \leq p^\ell -1$ we have
\[
c(\lambda,k)=
\left\{
\begin{array}{ll}
 \frac{1}{p^r} \left(g_{E_r^T}-1+ \Gamma_{k.\lambda}
 \right)
& \mbox{ if } k\geq p^\ell-p^r +1  \mbox{ or } \lambda\neq 0 \\
g_{F^G}
&
\mbox{ if } k=p^\ell-p^r \mbox{ and } \lambda =0.
\end{array}
\right.
\]
}

\end{proposition}
This will allow us to see:
\begin{theorem} \label{main-th} 
If $r=0$ then 
\begin{equation}\label{th1}
 d(\lambda,p^\ell)=g_{F^G}-1+\Gamma_{p^\ell,\lambda}+\Lambda_{p^\ell,\lambda}.
\end{equation}
For { all the values of $r$ and for} $k< p^\ell-p^r$ we have:
\[
 d(\lambda,k)=\Gamma_{k-1,\lambda}-\Gamma_{k,\lambda}+M_{k,\lambda},
\]
where $\{-1,0,1 \} \ni M_{k,\lambda}:=\Lambda_{k-1,\lambda}-\Lambda_{k,\lambda}$.

If $k=p^\ell-p^r$ then $k-1=p^\ell-p^r-1$ and 
\begin{eqnarray*}
 d(\lambda,p^\ell-p^r)&=&g_{F^G}-1+\Gamma_{k-1,\lambda}+\Lambda_{k-1,\lambda}-c(\lambda,p^\ell-p^r)\\
 &=&\left\{
\begin{array}{ll}
 \Gamma_{k-1,\lambda}-\frac{1}{p^r}\Gamma_{k,\lambda}+\Lambda_{k-1,\lambda}, &  \mbox{ if }\lambda \neq 0 \\
\Gamma_{k-1,0}+\Lambda_{k-1,0}-1, &\mbox{ if } \lambda=0
\end{array}
\right.
\end{eqnarray*}

For { $r\neq 0$ and} $p^\ell-p^r \leq k \leq p^\ell -1$, we have:
\begin{equation}\label{th2}
 d(\lambda,k)=\left\{
\begin{array}{ll}
 1 & \mbox{if } k=p^\ell-p^r+1 ,\lambda=0 
\\
 \frac{1}{p^r}\left(g_{E_r^T}-1 +\Gamma_{k, \lambda}\right)
 &  \mbox{if } k=p^\ell \\
 0 & \mbox{otherwise}
\end{array}
\right.
\end{equation}
\end{theorem}
\begin{proof}
 The proof is a simple application of proposition \ref{pro3}.
\end{proof}

\subsection{Computation of $\mathbf{c(\lambda,k)}$}
This section is devoted to the proof of proposition \ref{pro3}.

\begin{lemma} \label{TamagawaREP}
 Let $G$ be a group of order $p^\ell n$  acting on the curve $X$, with only tame ramification,
 i.e. every point that 
is ramified has decomposition group $G(P) \subset \langle g^{p^\ell} \rangle$.
Let $T=\langle g^{p^\ell} \rangle$ be the tame cyclic part of the group $G$. 
Consider the integers $\phi_i,\Phi_i,\alpha_\lambda,e_i'$ describing the Kummer extension 
$F/F^T$ and let $g_{F^T}$ denote the genus of $F^T$.
Then the decomposition of the space $V$ of holomorphic differentials is given by
\[
 V:=\bigoplus_{k=1}^{p^\ell} \bigoplus_{\lambda=0}^{n-1} V(\lambda,k)^{d^*(\lambda,k)},
\]
where 
$d^*(0,1)=1$,
 \[
  d^*(\lambda,p^\ell)=\frac{1}{p^\ell}\left(g_{F^T}-1+\sum_{i=1}^t \left\langle \frac{-\alpha_\lambda \Phi_i}{e_i'}\right\rangle\right)
 \]
and $d^*(\lambda,k)=0$ in all other cases. 
%
%
\end{lemma}
\begin{proof}
 Group actions  on curves without branched points on spaces of 
holomorphic differentials were studied by T. Tamagawa \cite{Tamagawa:51}.
 Tamagawa proved that the space of holomorphic differentials is decomposed as 
\[
 V:=K \oplus K[P]^{g_{X/P}-1},
\]
where $g_{X/P}$ is the genus of the quotient curve $X/P$. 

Actions with tame ramification where studied by E. Kani \cite{Kani:86}. Kani proved
that:
\[
 V:=K \oplus K[G]^{g_{X/G}-1} \oplus \tilde{R}^*_G,
\]
where $ \tilde{R}^*_G$ is a $k[G]$-module such that $n  \tilde{R}^*_G=R_G^*$ and 
$R_G^*$ is a the contragredient module of the tame ramification module (for precise definition see
\cite[sec. 1]{Kani:86}).

The result of Tamagawa for the action of the $p$-group $P=\langle g^n \rangle$ 
gives that 
\begin{equation} \label{TamagawaComp}
 V=K\oplus \bigoplus_{\lambda=0}^{n-1} V(\lambda,p^\ell)^{d^*(\lambda,p^\ell)}.
\end{equation}
The integers $d^*(\lambda,p^\ell)$ can be computed by a careful look at the definition of the 
tame ramification module. We will instead  compute them using the results of Valentini-Madan 
for the extension $F/F^T$, $T={\langle g^{p^\ell} \rangle}$.

The extension $F/F^T$ is a cyclic Kummer extension with 
Galois group generated by $\sigma=g^{p^\ell}$ and it is 
characterized by the integers $\phi,\Phi,e_i',\alpha_\lambda$ introduced in section \ref{FR}.
For the module of holomorphic differentials
the multiplicities $m_\lambda $ of the character $\lambda$ given by the action $\sigma^j (v)=\zeta^{\lambda j} v$
are  equal to 
\[m_\lambda=g_{F^T}-1+\sum_{i=1}^t \left\langle \frac{-\alpha_\lambda \Phi_i}{e_i'}\right\rangle, \mbox { if } \lambda \neq 0\]
and 
\begin{eqnarray*}
 m_0&=&g_{F^T}+\sum_{i=1}^t \left\langle \frac{-\alpha_0 \Phi_i}{e_i'}\right\rangle  \\
 &=&g_{F^T}, \mbox{ if } \lambda=0.
\end{eqnarray*}
{For the last equality it is enough to notice that as $(\Phi_i,e_i')=1$ then 
  $\sum_{i=1}^t \left\langle \frac{-\alpha_0 \Phi_i}{e_i'}\right\rangle=0$, since the condition 
$e_i' \mid \alpha_0$ is equivalent to $e_i' \mid 0$ for all $i$'s (see \cite[page 115]{vm}).}

From Tamagawa result we have that  $d^*(0,1)=1$,
 while for the remaining $m_0-1=g_{F^T}-1$  representations give us that in both cases 
($\lambda=0$ and $\lambda \neq 0$) we have 
 \[
  d^*(\lambda,p^\ell)=\frac{1}{p^\ell}\left(g_{F^T}-1+\sum_{i=1}^t \left\langle 
\frac{-\alpha_\lambda \Phi_i}{e_i'}\right\rangle\right)
 \]
and $d^*(\lambda,k)=0$ in all other cases. Notice that the eigenvalue $\zeta^\lambda$ appears
$p^\ell$ times in every component $V(\lambda,k)$. 
\end{proof}

{\begin{remark}
Applying Riemann-Hurwitz formula, we  obtain:
\[
  d^*(\lambda,p^\ell)=g_{F^G}-1+\frac{1}{p^\ell}\sum_{i=1}^t \left\langle 
\frac{-\alpha_\lambda \Phi_i}{e_i'}\right\rangle
 \]
\begin{itemize}
\item If $F/F^T$ is unramified, i.e. when $e_i'=1$ for all $i$, then this coincides with the result of Tamagawa, \cite{Tamagawa:51}.
\item If $p^\ell=1$, i.e. $F^G=F^T$, then this coincides with the result of Hurwitz  \cite[Theorem 3.5, p. 600]{MorrisonPinkham}, after letting  $r=1$,
to eq. (\ref{akdef}).
\end{itemize}
\end{remark}
}

Is there a place  $P$ of $F$ that is fully ramified in extension $F/F^{\langle g^n \rangle}$?
 If not then 
we consider the place $P$  with maximal ramification index. Set $r=\ell-\max\{\epsilon_i\}$.
The wild  decomposition 
group $\langle g^n \rangle (P)$ at this 
place is cyclic and we will denote  the corresponding fixed field by $E_r$.
Call $E$ the fixed field of the wild part $\langle g^n \rangle$. 
Then we will have a tower of fields $F/E_r/E$ such that in extension $E_r/E$ there 
is no ramification at all. Notice that $r=0$ and $E_r=E$ if and only if there is 
a place $P$ fully ramified in extension $F/F^{\langle g^n \rangle}$.

For the study of the spaces $V^{k+1}/V^k$, with $k=0, \ldots, p^\ell-1$, we will distinguish two cases:
\newline
{\bf Case 1.} ${k< p^\ell-p^r}$.

\begin{lemma} \label{le:6}
Assume that $k< p^\ell-p^r$.
 If the differential  $\omega=\sum_{\nu=0}^k c_\nu w_\nu dx \in V^{k+1}$, {representing a class in  $  V^{k+1}/ V^{k}$},
  is holomorphic then 
\[c_k \in L_{F^P}\left(\mathrm{div}_{F^P}(dx) +\sum_{\mu=1}^s \nu_{\mu,k} \bar{P}_\mu \right).\]
The space $V^{k+1}/V^k$ is of dimension $g_{F^P}-1+\sum_{\mu=1}^s \nu_{\mu,k}$. 
\end{lemma}
\begin{proof}
See the proof of theorem 1 and  page 112 in \cite{vm}.
\end{proof}

In order to study the $k[T]$-module structure of the space $V$ we will apply the previous argument 
with $f$ in place of the $c_k$ and we focus our  study to  the space of differentials
 which  have poles at $ \sum_{\mu=1}^s \nu_{\mu,k}P_\mu$, i.e. differentials of the form:
\begin{equation} \label{hol-rep}
  \omega= f dx \mbox{ such that }  \mathrm{div}_{F^P}(f dx) \geq -\sum_{\mu=1}^s \nu_{\mu,k}\bar{P}_\mu.
\end{equation}
We may choose the function $x \in F^P$ to be a function in $F^G$. Let $\tau=g^{p^\ell}$ be a
generator of the cyclic group $T$. Recall that we assumed that $\tau y=\zeta_n^r y$ and we 
have selected $\alpha_\lambda$ such that $r \alpha_\lambda=\lambda \mod n$.  Assume  that 
\[
 \tau (f dx )= \zeta^\lambda f dx.
\]
 By eq. (\ref{akdef}) we have
\[
 \tau(f/y^{\alpha_\lambda})=f/y^{\alpha_\lambda},
\]
so $f=h y^{\alpha_\lambda}$ with $h\in F^G$. 
Therefore, eq. (\ref{hol-rep}) is satisfied if and only if
\begin{equation} \label{con-tame-down}
 N_{F^P/F^G}( \mathrm{div}( f dx)) \geq N_{F^P/F^G}\left( -\sum_{\mu=1}^s \nu_{\mu,k} \bar{P}_\mu\right).
\end{equation}
We compute:
\[
 N_{F^P/F^G}( \mathrm{div}(f dx))=
 n\mathrm{div}_{F^G}{h}+
  \alpha_k \mathrm{div}_{F^G}(b)+
n\mathrm{div}_{F^G}(dx)+
\mathrm{D}(F^P/F^G).  
\]
\begin{remark}\label{degrees}
Whenever we write down a reduced divisor 
 $A=\sum \alpha_i  P_i$ (i.e. $P_i\neq P_j)$ with $q_i \in \mathbb{Q}$ we 
mean the divisor $\sum \lf q_i \rf P_i$.
Notice that if  $A=\sum \alpha_i P_i$ is a divisor (with possible rational  coefficients) and $B=\sum 
\beta_j P_j$ is a divisor
with integer coefficients, then since for $\alpha \in \mathbb{Q}, \beta\in 
\mathbb{Z}$ $\lf \alpha+\beta \rf=\lf \alpha \rf +\beta $ we
have that 
\[
 A+B=\sum \lf \alpha_i \rf P_i + \sum\beta_j P_j,
\]
 i.e. we don't have to write down $A+B$ in reduced form, before taking the integral part of 
its coefficients. 
\end{remark}
Using eq. (\ref{div-b}),(\ref{diff-tame}) we see that 
eq. (\ref{con-tame-down}) is equivalent to 
\[
 \mathrm{div}(h) \geq -\mathrm{div}_{F^G}(dx)
- \alpha_\lambda A
-\sum_{j=1}^t \left( \frac{\alpha_\lambda \phi_j}{n} +1 -\frac{1}{e_j'}\right) \bar{Q}_j
 -\frac{1}{n}N_{F^P/F^G}\left( \sum_{\mu=1}^s \nu_{\mu,k}\bar{P}_\mu\right),
\]
i.e. $h \in L(W+E_{k,\lambda})$. Notice that the norm $N_{F^P/F^G}(P_\mu)$ is just the place of $F^G$ 
 lying below $P_\mu$. 
We proved the following 
\begin{lemma}\label{degE}
The subspace of $V^{k+1}/V^k$ of elements where $g$ acts by multiplication by $\zeta^\lambda$ is isomorphic to 
the space  
 $L_{F^G}(W+E_{k,\lambda})$, 
where $W$ is a canonical divisor on $F^G$ and 
\[
 E_{k,\lambda}:=\alpha_\lambda A
+\sum_{j=1}^t \left( \frac{\alpha_\lambda \phi_j}{n} +1 -\frac{1}{e_j'}\right) \bar{Q}_j
 +\frac{1}{n}N_{F^P/F^G}\left( \sum_{\mu=1}^s \nu_{\mu,k}P_\mu\right)
\]
is an effective divisor.
\end{lemma}

%
%

We will now write $E_{k,\lambda}$ as a sum of an integral divisor and of a divisor in reduced form.
We can assume that $\{\bar{Q}_1,\ldots, \bar{Q}_{t_0} \}$ is the set of ramified places
such that  their extensions in $F^P$ do not ramified further in $F/F^P$.
We will denote by $\{ \bar{Q}_{t_0+1},\ldots,\bar{Q}_t\}$ the rest of the ramified places.
For the $s-(t-t_0)$ places of $F^P$ that are not ramified in 
$F^P/F^G$ we  will denote by $\Pi_\mu$ the places $\sum_{j=1}^n \tau P_\mu$. 

Now $E_{k,\lambda}$ can be  written:
\begin{eqnarray*}
 E_{k,\lambda}&:=&\alpha_\lambda A +\sum_{j=1}^{t_0}
 \left( \frac{\alpha_\lambda \phi_j}{n} +\frac{e_j'-1}{e_j'}\right) \bar{Q}_j
+ \sum_{j=t_0+1}^{t} \left( \frac{\alpha_\lambda \phi_j+\nu_{j,k}}{n} +\frac{e_j'-1}{e_j'}\right) \bar{Q}_j\\
& & +
\sum_{\mu=1}^{s-t+t_0}
\frac{\nu_{\mu,k}}{n} \Pi_\mu.
\end{eqnarray*}
The divisor  $E_{k,\lambda}$ as it is written above is not necessarily in reduced form. 
We don't know whether the divisor 
$A$ is prime to $\bar{Q}_i$ or $\Pi_\mu$. But since it has integer coefficients
{ and since all the divisors with possibly rational coefficients are prime to each other}, we arrive 
at 
\begin{eqnarray} \nonumber
 E_{k,\lambda}& :=& \alpha_\lambda A +\sum_{j=1}^{t_0} \lf \frac{\alpha_\lambda \phi_j}{n} +\frac{e_j'-1}{e_j'}\rf \bar{Q}_j
+ \sum_{j=t_0+1}^{t} \lf \frac{\alpha_\lambda \phi_j+\nu_{j,k}}{n} +\frac{e_j'-1}{e_j'}\rf \bar{Q}_j
\\ & & +
\sum_{\mu=1}^{s-t+t_0}
\lf
\frac{\nu_{\mu,k}}{n}
\rf \Pi_\mu. \label{finE}
\end{eqnarray}

\begin{lemma} \label{degE_k}
 The degree of $E_{k,\lambda}$ equals to
{
\[
 \mathrm{deg}(E_{k,\lambda}):=\sum_{i=1}^{t} \left\langle -\frac{\alpha_\lambda \Phi_i}{e_i'}\right\rangle+
\sum_{j=t_0+1}^{t}
\lf
\left\langle
\frac{\alpha_\lambda \Phi_j -1}{e_j'}
\right\rangle
+
\frac{\nu_{j,k}}{n}
\rf
+
\sum_{\mu=1}^{s-t+t_0}
\lf
\frac{\nu_{\mu,k}}{n}
\rf.
\]
}
\end{lemma}
\begin{proof}
Following Valentini Madan we see that 
that 
\begin{equation}\label{degA}
\deg(A)=-\sum_{j=1}^t \frac{\phi_j}{n}=-\sum_{j=1}^t \frac{\Phi_j}{e_j'}
\end{equation} 
(recall that $\Phi_i=\phi_i/(\phi_i,n)$). 
%
\begin{eqnarray*}
 \deg E_{k,\lambda} &= &\sum_{j=1}^{t_0} 
\left\langle -\frac{\alpha_\lambda \Phi_j}{e_j'} \right\rangle
+ \sum_{j=t_0+1}^{t}
\left( 
\lf \frac{\alpha_\lambda \phi_j+\nu_{j,k}}{n} +\frac{e_j'-1}{e_j'}
\rf-\frac{\alpha_\lambda \Phi_j}{e_j'}
\right)
 \\  & &+ 
\sum_{\mu=1}^{s-t+t_0}
\lf
\frac{\nu_{\mu,k}}{n}
\rf.
\end{eqnarray*}
We will use
\[
 \frac{\alpha_\lambda \Phi_j}{e_j'}=\lf \frac{\alpha_\lambda \Phi_j}{e_j'} \rf+ \left\langle \frac{\alpha_\lambda \Phi_j}{e_j'} \right\rangle, 
\]
We have
{
\[
\lf \frac{\alpha_\lambda \phi_j+\nu_{j,k}}{n} +\frac{e_j'-1}{e_j'}
\rf-\frac{\alpha_\lambda\Phi_j}{e_j'}=
 \]
 \[
= \left\langle -\frac{\alpha_\lambda \Phi_j}{e_j'}\right\rangle +
\lf
\frac{\alpha_\lambda \Phi_j +e_j'-1}{e_j'}+\frac{\nu_{j,k}}{n}
\rf 
+\lf -\frac{\alpha_\lambda \Phi_j}{e_j'} \rf= \]
\[=\left\langle -\frac{\alpha_\lambda \Phi_j}{e_j'}\right\rangle +
\lf 
\lf 
\frac{\alpha_\lambda \Phi_j +e_j'-1}{e_j'} 
\rf +
\left\langle \frac{\alpha_\lambda \Phi_j +e_j'-1}{e_j'}\right\rangle +\frac{\nu_{j,k}}{n}
\rf +
\lf -\frac{\alpha_\lambda \Phi_j}{e_j'} \rf= \]
\[= \left\langle -\frac{\alpha_\lambda \Phi_j}{e_j'}\right\rangle +
\lc
\frac{\alpha_\lambda \Phi_j}{e_j'}
\rc
+
\lf
\left\langle \frac{\alpha_\lambda \Phi_j +e_j'-1}{e_j'}\right\rangle +\frac{\nu_{j,k}}{n}
\rf
+
\lf -\frac{\alpha_\lambda \Phi_j}{e_j'} \rf =\]
\[=\left\langle -\frac{\alpha_\lambda \Phi_j}{e_j'}\right\rangle 
+
\lf
\left\langle \frac{\alpha_\lambda \Phi_j -1}{e_j'}\right\rangle +\frac{\nu_{j,k}}{n}
\rf.
\]
}
\end{proof}

\begin{proposition} \label{comp-c}
{ If $k< p^\ell -p^r$,} we have
\[
 c(\lambda,k)=\dim L(W+E_{k,\lambda})=g_{F^G}-1+ \deg(E_{k,\lambda}) +\Lambda_{k,\lambda}
\]
Moreover, if 
\[
 \deg(E_{k,\lambda})=0 
\]
then  $\Lambda_{k,\lambda}=1$. In all other cases $\Lambda_{k,\lambda}=0$.
\end{proposition}
\begin{proof}
By Riemann-Roch theorem and lemma \ref{degE_k} we see that 
{
\[
 \dim L(W +E_{k,\lambda})=g_{F^G}-1+ \deg E_{k,\lambda}+\dim L(-E_{k,\lambda}).
\]
}
If the  divisor $\deg(E_{k,\lambda})>0$  then $\dim L(-E_{k,\lambda})=0$ and the result follows.

Assume now that $\deg(E_{k,\lambda})=0$. Since $E_{k,\lambda}$ is effective this means that $E_{k,\lambda}=0$
and in this case $\Lambda_{k,\lambda}=\ell(0)=1$. 
\end{proof}

{\bf Case 2.} $p^\ell-p^r \leq  k\leq p^\ell-1 $.
In this case we will apply the same procedure as we did in Case 1 and then we will apply 
lemma \ref{TamagawaREP} for the extension $E_r/F^G$. 
Write $k=k_1+t p^r$ with $0 \leq k_1 <  p^r$
 and $t=p^{\ell-r}-1$. Set $k_0=tp^r=p^\ell-p^r$.
Let $\sigma=g^{n}$ be a generator for the $p$-cyclic part of $G$.
Let $V_{E_r}^{k}$ be the space of holomorphic differentials of $E_r$ that are 
annihilated by $(\sigma-1)^k$. 
Valentini-Madan \cite[p.111-112]{vm} proved that
 \begin{equation}\label{isomorphism}
(\sigma-1)^{k_0}:V^{k+1}/V^k \rightarrow V_{E_r}^{k_1+1}/V_{E_r}^{k_1}
 \end{equation}
is an isomorphism. We will now consider the extension $E_r/F^G$ and we will apply lemma 
\ref{TamagawaREP} in order to compute the decomposition into indecomposable 
$G/\langle g^{p^r}\rangle$-modules. 
Let $c^*(\lambda,k_1)$ be the number of characters $\zeta \mapsto \zeta^\lambda$ in 
the module $V_{E_r}^{k_1+1}/V_{E_r}^{k_1}$.
We compute{ that $c(\lambda,k)$ equals to}:
%
\begin{eqnarray*}
 c(\lambda,k_1+p^\ell-p^r) &= & c^*(\lambda,k_1) \\
 & = &\sum_{\mu \geq k_1+1} d^*(\lambda,\mu)=
\left\{
\begin{array}{ll}
 d^*(\lambda,p^\ell), & \mbox{if }  k_1\geq 1   \mbox{ or }\lambda \neq 0 \\
 d^*(0,p^\ell)+1 { =g_{F^G}}, & \mbox{if } k_1=0,\mbox{ and } \lambda=0
\end{array}
\right.
\end{eqnarray*}
Therefore, for $k=p^\ell$ we compute:
{
\[
 d(\lambda,p^\ell)=c(\lambda,p^\ell-1)=d^*(\lambda,p^\ell)=
\frac{1}{p^r}\left(g_{E_r^T}-1+\sum_{i=1}^t \left\langle \frac{-\alpha_\lambda \Phi_i}{e_i'}\right\rangle\right)
\]
}
by lemma \ref{TamagawaREP}.
Moreover { for $p^\ell-p^r \leq  k\leq p^\ell-1 $} and from eq. (\ref{dk}) and the isomorphism given in eq. (\ref{isomorphism}) we obtain:
\begin{eqnarray*}
 d(\lambda,k) & = & c(\lambda,k-1)-c(\lambda,k), \\
        & = & c^*(\lambda,k-1-\bigl(p^\ell-p^r\bigr))-c^*(\lambda,k-\bigl(p^\ell-p^r\bigr))\\
        &=  & d^*(\lambda,k-\bigl(p^\ell-p^r\bigr))=0
\end{eqnarray*}
unless $k=p^\ell-p^r+1$ and $\lambda=0$. In this case $d^*(0,1)=1=d(0,p^\ell-p^r+1)$.

%
%
%
%

{
\begin{remark}
Notice that   when $k\geq p^\ell-p^r$, then $\nu_{\mu,k}=0$ (see also \cite[p. 110]{vm}).
Thus Boseck invariants (Definition \ref{defBoseck}) take now the form 
\[
\Gamma_{k, \lambda} =\Gamma_{\lambda}=\sum_{i=1}^t\left\langle\frac{-\alpha_\lambda \Phi_i}{e_i'}\right\rangle.
\]
With this in mind, we take that
\[
 d(\lambda,p^\ell)=
\frac{1}{p^r}\left(g_{E_r^T}-1+\Gamma_{ k, \lambda}\right).
\]
\end{remark}
This completes the proof of theorem \ref{main-th}.
}
{
\begin{example}
Suppose that $F^G=F^P$, i.e. $n=1$ then,
from eq. (\ref{th1}) and (\ref{th2}) respectively, we get that the regular representation of $G$ occurs
\[
 d(\lambda,p^\ell)= d(p^\ell)=
 \left\{
\begin{array}{ll}
g_{F^G}-1,  \mbox{ if }r \neq 0\\
 g_{F^G},\mbox{ otherwise } 
\end{array}
\right.
\]
times in the representation of $G$ in $V$. This result coincide with the results obtained in \cite{vm}.
\end{example}
}
%

\section{Relation to the theory of Weierstrass semigroups}

Aim of this section  is to find a relation between the Galois module structure 
of the space of holomorphic differentials and the Weierstrass semigroup attached to a 
ramified point. 
In characteristic zero there are results \cite{MorrisonPinkham} 
relating the structure of the Weierstrass 
semigroup at $P$  to the subgroup $G(P)$. 
For example there is  a  theorem due to J. Lewittes \cite[Th. 5]{Lewittes63}
which  relates the structure of the semigroup to the  module structure of 
holomorphic differentials. Also I. Morrison and H. Pinkham 
\cite{MorrisonPinkham}
considered the case of Galois Weierstrass points, i.e. covers of the form $X\rightarrow \mathbb{P}^1$ 
with cyclic cover group in characteristic $0$. 

Let us start with  a  convenient description of a semigroup:
Let $\Sigma\subset \mathbb{N}$ be a semigroup and let $d$ be the least positive number in 
$\Sigma$. For 
$1\leq i \leq d-1$ we denote by $b_i$ the smallest element in $\Sigma$ congruent to 
$i \mod d$, and define $\nu_i$ by the equation:
\begin{equation} \label{semnu}
 b_i =\nu_i d +i,
\end{equation}
i.e. $\nu_i=\lf\frac{d}{b_i} \rf$.
The numbers $\nu_i$ equals to  the number of gaps $a_k$ for which $a_k\equiv i \mod d$, 
and the semigroup $\Sigma$ is characterized by them.

From now on $\Sigma$ will be the Weierstrass semigroup attached to a point $P$.
Let $f$ be a function on $X$ such that $(f)_\infty=d P$. This gives rise to a map 
$f:X \mapsto \mathbb{P}^1$ and we assume that this map is a Galois cover with Galois group $G(P)$.
In characteristic $0$ the group $G(P)$ is always cyclic and the space of holomorphic differentials is 
described by a theorem due to Lewittes and Hurwitz \cite[th. 1.3, th 3.5]{MorrisonPinkham}. 

In this paper we  assume that  $G(P)$ is a cyclic group of order  $n p^\ell$.
The following theorem is a generalization of the theorems of Lewittes and Hurwitz 
 written in the language of Brauer characters \cite{SerreLinear}.
\begin{proposition} \label{r1}
Let $T$ be the tame cyclic part of $G(P)$.
Let $L$ be a complete local ring that contains the $n$-th roots of unity,
let $\mathcal{O}_L$ be its valuation ring and let $m_L$ be the maximal 
ideal of $\mathcal{O}_L$ such that $\mathcal{O}_L/m_L=K$.
For example we can take $L=W(K)[\zeta_n]$.
The modular character of $\mu: G_{\mathrm{reg}}\rightarrow \mathcal{O}_L$ induced by the 
$K[G]$-module of holomorphic differentials can be written as
\[
 \mu=\sum_{i=1}^{d-1} \mu_i \chi^i,
\]
where $\chi$ is a generator of the character group $\hat{G}_{\mathrm{reg}}$
 of the cyclic group $G_{\mathrm{reg}}=\Z/n\Z$ and 
$\mu_i$ are equal to the number of gaps at $P$ that are equivalent to $i \mod n$.
\end{proposition}
\begin{proof}
The proof we will write is a modification of the characteristic zero 
proof given in \cite[th. 1.3]{MorrisonPinkham}.

By construction for every $\sigma \in T$ we have $\sigma(T)=T$. Let $\tau$ be 
the generator of $T$. By the lemma of Hensel we might assume that 
there is a local uniformizer $t$ at $P$ such that 
$\sigma(t)=\zeta t$ where $\zeta$ is a primitive $n$-th root of unity. 

Every gap $a_k$ corresponds by Riemann-Roch theorem to a holomorphic 
differential $\omega_k$ with a root of order $a_k-1$ at $P$. 
Observe that the flag of vector spaces $\langle \omega_g,\ldots,\omega_k \rangle$
are invariant under the action of $G$ and by a trigonal change of 
coordinates we might assume that $\omega_k$ can be selected in such a way 
so that $\omega_k =t^{a_k-1} dt$. For this selection of  $\omega_k$ we have that 
$\tau \omega_k=\zeta^{a_k} \omega_k$ and the result follows. 

\end{proof}

%

This proposition does not describe completely the relation of the semigroup 
and the $K[G]$-module structure since it gives information of the number 
gaps modulo $n$ and not modulo $np^\ell$ as required. {Notice that by construction 
$d:=np^\ell$ is the smallest non zero pole number.
} 
\begin{definition}
For every $i$, $0\leq i < n p^\ell$ we consider the reductions of
of $i$ modulo $p^\ell$ and $n$ respectively, namely:
$i_0=i \mod n$  and 
$i_1=i  \mod p^\ell$. We will denote by 
\[
 \bar{c}(i_0,i_1)=\mbox{ the number of gaps  at $P$ }  \mbox{ that are equivalent to } i \mod n p^\ell.
\]
\end{definition}
Of course these quantities are related to the $\mu_i$ defined in proposition \ref{r1}.
For an $i_0$ with $0\leq i_0 < n$ we have 
\[
\mu_{i_0} = 
\sum_{i_1=0}^{p^\ell-1} \bar{c}(i_0,i_1). 
\]
We will give an independent and complete 
description in terms of the decomposition given in eq. (\ref{decompo-1}).

\begin{proposition} \label{numb-gaps}
Let $r_{\mu,k}$ be the remainder  of the division of $\delta_\mu+v_{P_{\mu,\nu}}(w_k)$ by $p^{\ell}$,
 i.e.
\begin{equation} \label{rem}
 r_{\mu,k}=\delta_\mu+v_{P_{\mu,\nu}}(w_k)-p^{\ell}\nu_{\mu,k}.
\end{equation}
The holomorphic differentials
in $V^{k+1}$
have roots at $P_{\mu,\nu}$ of orders
\begin{equation} \label{1idiot}
r_{\mu,k}+p^{\ell}\xi, \mbox{ with }  \xi \in \{0,1,\ldots, \sum_{\mu=1}^s \nu_{\mu,k}-2\} 
\cup B_{\mu,k},
\end{equation}
where $B_{\mu,k}$ is a subset of natural numbers with $g_{F^P}$ elements, 
all  greater  than $\sum_{\mu=1}^s \nu_{\mu,k} $.
The dimension of the space of  holomorphic differentials in $V^{k+1}$ that have roots of order $x$  such that:
\begin{eqnarray} \label{cond11}
 x & \equiv & r_{\mu,k} \mod p^{\ell},\\
 x & \equiv & \alpha_\lambda \mod n \nonumber 
\end{eqnarray}
is equal to $c(\lambda,k)=\bar{c}(a_\lambda+1,r_{\mu,k}+1)$. 
\end{proposition}
\begin{proof}
By lemma \ref{le:6} the differential  $\sum_{\nu=0}^k c_\nu w_\nu dx$ is holomorphic if 
 the 
 elements $c_k$ are in  $\mathcal{L}:=L(\mathrm{div}_{F^P}(dx) + \sum_{\mu=1}^s \nu_{\mu,k} \bar{P}_{\mu}  )$.
What are the possible valuations of such elements at a fixed $\bar{P}_{\mu_0}$?
Fix $\mu_0$ and 
consider the divisors:
\[
 A_{j}:=\mathrm{div}_{F^P}(dx) + \sum_{\mu=1,\mu\neq \mu_0}^s \nu_{\mu,k} \bar{P}_{\mu} +j \bar{P}_{\mu_0},
\]
for $j< \nu_{\mu_0,k}$. We have that $L(A_j) \subset L(A_{\nu_{\mu_0,k}})$. There is an 
element $c_k$ with $v_{\bar{P}_{\mu_0}}(c_k)=-v_{\bar{P}_{\mu_0}}(\mathrm{div}_{F^P}(dx))-j$ if and only if 
$\ell(A_j)-\ell(A_{j-1})=1$:
Indeed, by using Riemann-Roch theorem we see that 
\[
 \ell(A_j)=g_{F^P}-1+\sum_{\mu=1,\mu\neq \mu_0}^s \nu_{\mu,k}+j+
\ell\left(
-\sum_{\mu=1,\mu\neq \mu_0}^s \nu_{\mu,k} \bar{P}_{\mu} -j \bar{P}_{\mu_0}
\right).
\]
Therefore, 
\[
 \mbox{if} \sum_{\mu=1,\mu\neq \mu_0}^s \nu_{\mu,k}+j-1\geq 0 \mbox{ then } \ell\left(
-\sum_{\mu=1,\mu\neq \mu_0}^s \nu_{\mu,k} \bar{P}_{\mu} -j \bar{P}_{\mu_0}
\right)=0
\]
and 
$\ell(A_j)-\ell(A_{j-1})=1$ and there is an element $c_k$ with valuation at $\bar{P}_{\mu_0}$ equal to {$-v_{\bar{P}_{\mu_0}}(\mathrm{div}_{F^P}(dx))-j$}. 
This proves that possible valuations $v=-v_{\bar{P}_{\mu_0}}(\mathrm{div}_{F^P}(dx))-j$ of elements in $\mathcal{L}$ at  $\bar{P}_{\mu_0}$ satisfy
\[
 -\nu_{\mu_0,k}-v_{\bar{P}_{\mu_0}}(\mathrm{div}_{F^P}(dx)) \leq v \leq \sum_{\mu=1,\mu\neq \mu_0}^s \nu_{\mu,k}-1-v_{\bar{P}_{\mu_0}}(\mathrm{div}_{F^P}(dx)),
\]
i.e.
\begin{equation} \label{rem1}
 0 \leq \nu_{\mu_0,k}-j \leq \sum_{\mu=1}^s \nu_{\mu,k}-1.
\end{equation}
The valuation at $P_{\mu_0}$ of the  differential $c_k w_k dx$ of $F$ equals:
\begin{equation} \label{rem2}
 p^{\ell} v_{\bar{P}_{\mu_0}}(c_k dx) + \delta_{\mu_0}+v_{P_{\mu_0,\nu}}(w_k).
\end{equation}
Recall that $\delta_{\mu_0}+v_{P_{\mu_0,\nu}}(w_k)=r_{\mu_0,k}+p^{\ell} \nu_{\mu_0,k}$ by eq.  (\ref{rem}), so 
(\ref{rem2}) becomes 
\[
 p^{\ell} v_{\bar{P}_{\mu_0}}(c_k dx)+r_{\mu_0,k}+p^{\ell} \nu_{\mu_0,k}
\]
which in turn by using (\ref{rem1}) implies  
that  the possible valuations of differentials in $V^{k+1}$ contain the set 
\[
 r_{\mu_0,k}+p^{\ell}\xi, 0\leq \xi< \sum_{\mu=1}^s \nu_{\mu,k}-1.
\]
If $g_{F^P}=0$ then there are no other possible valuations for the elements $c_k$
since the above valuations are different, the corresponding functions are 
linear independent and have the correct dimension given in lemma \ref{le:6}. 
If $g_{F^P}>0$ then there are 
 $g_{F^P}$ more possible valuations, but their exact values can not be easily described.
Indeed,
 notice that always $\ell(A_j)-\ell(A_{j-1}) \leq 1$ by  \cite[I.4.8]{StiBo}.

Using lemma \ref{degE} and proposition \ref{comp-c} we compute that the dimension of the differentials satisfying the 
conditions given in (\ref{cond11}), is equal to  
 $c(k,\lambda)$.

It is a well known application of the Riemann-Roch theorem that 
the existence of a differential with root of order $a-1$ at $P$ implies 
that $a$ is a gap at $P$. 
 Therefore if we add 1 to   the natural numbers 
appearing in eq. (\ref{1idiot}) then we obtain   all the gaps at $P$ coming from holomorphic differentials
 in $V^{k+1}$. 
All of them are equivalent to $r_{\mu,k}+1$  modulo $p^\ell$. 
Moreover, if $a$ is a gap at $P$ then there is a onedimensional subspace 
of $V$ such that the action of the tame part is given by 
$\zeta \mapsto \zeta^a$ \cite[th. 1.3]{MorrisonPinkham}. This proves the 
equality 
$c(\lambda,k)=\bar{c}(a_\lambda+1,r_{\mu,k}+1)$.
%
%
%
\end{proof}
{
\begin{remark}
Notice, that  now we are able to describe completely $\Sigma$ at $P_{\mu,\nu}$ by the method introduced by 
Morrison and H. Pinkham \cite{MorrisonPinkham} and explained in eq. (\ref{semnu}), when $P_{\mu,\nu}$
 ramifies completely. Indeed :
\begin{enumerate}
\item  The numbers $c(\lambda,k)=\bar{c}(a_\lambda+1,r_{\mu,k}+1)$ equal to  the number of gaps $x+1$ for which $x+1\equiv i \mod d$
and thus from the Chinese remainder theorem are equivalent to $\alpha_\lambda+1$ (or equivalently, see (2),  to $\lambda+1$) 
 and $r_{\mu,k}+1$ modulo $n$ and $p^\ell$ respectively.
\item  $r_{\mu,k}$ forms a complete system modulo $p^{\ell}$ as $k$ takes all the values $0,\ldots, p^\ell-1$,
and thus takes all the values from 1 to $d-1$. Moreover, let  $r=1$
to eq. (\ref{akdef}) (we use this argument widely through this paper). Then, in the same way we see that
 $\alpha_\lambda$ forms a complete system modulo $n$, as $\lambda$ runs through $0,\ldots, n-1$.
\end{enumerate}
\end{remark}
}
\section{The case of a cyclic $p$-group}
%
%
%
We will now focus on the case of cyclic extensions of the rational function field of order $p^\ell$.
We will also assume that every ramified place is ramified completely.
In this case we construct explicitly a basis of holomorphic differentials 
as follows:

We denote the ramified places of $K(x)$, by  $\bar{Q}_i=(x-\alpha_i),\ 1\leq i\leq s$,
 since in a rational function field  every ramified place 
corresponds to an irreducible polynomial, which is linear since the 
 field $K$ is algebraically closed.   We set
\begin{eqnarray*}
 g_k (x)=\prod_{i=1}^{s} (x-\alpha_i)^{\nu_{ik}}.
\end{eqnarray*}
\begin{definition}\label{cyclic}
For $k=0,1,\ldots,p^\ell -1,$ we define
\[
\Gamma_k:=\sum_{i=1}^s \nu_{ik}.
\]
\end{definition}

\begin{proposition} \label{holbasis-cyc} 
Let $X$ be a cyclic extension of degree $p^\ell$ of the rational function field.
The set 
\[
\left\{
\omega_{k\nu}^{(\alpha_i)} =(x-\alpha_i)^{\nu^{(k)}} g_{k}(x)^{-1}w_{k}dx: 0 \leq \nu^{(k)} \leq \Gamma_k -2, 0\leq k \leq p^n-2\right\}
\]
forms a basis for the set of holomorphic differentials for a cyclic extension of the 
rational function field of order $p^\ell$.
\end{proposition}
\begin{proof}
 We take the basis of \cite[Lemma 10]{karan}, set $m=1$ and modify it in order to evaluate
 holomorphic differentials in the ramified primes of the extension.
The same construction is given by Garcia in \cite[Theorem 2, Claim]{garciaelab} 
where  the  elementary abelian, totally ramified case is studied. 
The proof is identical to the one given there.
\end{proof}

Keep in mind that the natural number $i$ is a gap at $P$ if and only if there is a holomorphic differential 
$\omega$ with root at $P$ of order $i-1$.

\begin{lemma} \label{resi-li}
 The remainders  $r_{\mu,i}$ for different values of $i$ are different modulo $p^{\ell}$ and 
form a full set of representatives modulo $p^\ell$.
\end{lemma}
\begin{proof}
 Observe first that the valuations of the functions $w_k$ as $k$ runs over $0,\ldots,p^\ell-2$ are 
all different, since 
\[
 v_{P_{\mu,\nu}}(w_k)=-\sum_{j=1}^\ell a_j^k \Phi(\mu,j) p^{\ell-j}.
\]
Therefore the values $\delta_\mu+v_{P_{\mu,\nu}}(w_k)=-\sum_{j=1}^\ell a_j^k \Phi(\mu,j) p^{\ell-j}$ take 
all possible values modulo $p^\ell$.
\end{proof}


\begin{definition}
 For every natural number $0\leq a < p^\ell$ define by $\psi(a)$ the natural number such that 
\[
 r_{\psi(a),\mu}=a.
\]
Such a number exists by lemma \ref{resi-li}. 
\end{definition}

%

\begin{remark}
Recall that $r_{\mu,k}$ was defined in eq. (\ref{rem}) to denote the remainder of the 
division of $\delta_\mu+v_{P_{\mu,\nu}}(w_k)$ by $p^\ell$.
Boseck in his seminal paper \cite[Satz 18]{boseck},  where the $G=\Z/p\Z$ case is studied, states that as  
as $k$ takes all the values $0\leq k \leq p-2$ the remainder of the Boseck's basis 
construction $r_{\mu,k}$ takes all the 
values $0\leq r_{\mu,k} \leq p -2$  and thus all the numbers $1,\ldots,p-1$ are gaps.
 This is not entirely correct as we will show in example 
\ref{31acc}. The problem appears if there  is exactly  one ramified place in the Galois 
extension.  
\end{remark}

\begin{lemma} \label{smallGaps}
 If all $\Gamma_k \geq 2$ then all numbers  $1,\ldots,p^\ell-1$ are gaps.
If  there exist  Boseck invariants $\Gamma_k=1$, then the set of gaps smaller than 
$p^\ell$ is exactly the set $\{r_{\mu,k}: 0 \leq k \leq p^\ell-2, \Gamma_k \geq 2\}$.
\end{lemma}
\begin{proof}
As $k$ runs in $0\leq k \leq p^\ell-2$ the $r_{\mu,k}$ run in $0,\ldots,p^\ell-2$.
But the $\Gamma_k$ that are equal to $1$ have to be excluded since they give not rise 
to a holomorphic differentials in proposition \ref{holbasis-cyc}, see \cite[Eq. (21)]{karan} 
and  example \ref{31acc}.
\end{proof}
\begin{remark}
Notice that elements $\Gamma_k=1$ can appear only for primes
 $p \geq \Phi(\mu,j)$ and 
only if there is only one ramified place. 
\end{remark}

\begin{example} \label{31acc}  We consider the now the case of an Artin-Schreier extension of the 
function field $k(x)$, of the form $y^p-y=1/x^m$. 
In this extension only the place $(x-0)$ is ramified with different exponent $\delta_1=(m+1)(p-1)$. 
The Boseck invariants in this case are 
\[
 \Gamma_k=\lf \frac{(m+1)(p-1)-km}{p} \rf \;\; \mbox{ for } k=0,\ldots,p-2.
\]
The Weierstrass semigroup is known \cite{StiII} to be  
$
m \Z_+ + p \Z_+$. 
Let us now find the small gaps by using lemma \ref{smallGaps}.
If $p<m$ then all numbers $1,\ldots,p-1$ are gaps. If $p > m$ then $m$ is a pole number smaller 
than $p$.  
Indeed, $\Gamma_{p-2}=1$ and the remainder of the division of 
$(m+1)(p-1)-(p-2)m$ by $p$ is $r_{p-2}=m-1$. But then 
$r_{p-2}+1=m$ is not a gap.  
\end{example}

 \def\cprime{$'$}
\providecommand{\bysame}{\leavevmode\hbox to3em{\hrulefill}\thinspace}
\providecommand{\MR}{\relax\ifhmode\unskip\space\fi MR }
\providecommand{\MRhref}[2]{%
  \href{http://www.ams.org/mathscinet-getitem?mr=#1}{#2}
}
\providecommand{\href}[2]{#2}

%
\end{document}